\documentclass[11pt, 14paper,reqno]{amsart}
\usepackage{amsmath,amsfonts,amssymb}
\usepackage{algorithm}
\usepackage{booktabs}
\usepackage{multirow}

\usepackage{algpseudocode}
\usepackage{bigints}
\usepackage{longtable}
\usepackage{color}
\usepackage{comment}
\theoremstyle{plain}
\newtheorem{theorem}{Theorem}

\theoremstyle{proof}
\theoremstyle{definition}

\theoremstyle{remark}
\newtheorem{remark}{Remark}
\theoremstyle{lamma}

\numberwithin{equation}{section}
\numberwithin{lemma}{section}
\numberwithin{theorem}{section}
\numberwithin{remark}{section}
\numberwithin{prop}{section}
\numberwithin{corollary}{section}

\usepackage{amsmath}
\usepackage{amsfonts}   
\usepackage{amssymb}
\usepackage{amssymb, amsmath, amsthm}
\usepackage[breaklinks]{hyperref}

\theoremstyle{thmrm}

\numberwithin{conjecture}{section}

\begin{document}
\title[Computation of Jacobi sums of orders \MakeLowercase{$l^{2}$} and  \MakeLowercase{$2l^{2}$} with prime \MakeLowercase{$l$}]{\large{Computation of Jacobi sums of orders \MakeLowercase{$l^{2}$} and  \MakeLowercase{$2l^{2}$} with prime \MakeLowercase{$l$}}}
\author{Md Helal Ahmed, Jagmohan Tanti and Sumant Pusph}
\address{Md Helal Ahmed @ Department of Mathematics, Central University of Jharkhand, Ranchi-835205, India}
\email{ahmed.helal@cuj.ac.in}

\address{Jagmohan Tanti @ Department of Mathematics, Central University of Jharkhand, Ranchi-835205, India}
\email{jagmohan.t@gmail.com}

\address{Sumant Pushp @ Department of Computer Science and Technology, Central University of Jharkhand, Ranchi-835205, India.}
\email{sumantpusph@gmail.com}

\maketitle
\begin{abstract}
In this article, we present the fast computational algorithms for the Jacobi sums of orders $l^2$ and $2l^{2}$ with odd prime $l$ by formulating them in terms of the minimum number of cyclotomic numbers of the corresponding orders. We also implement two additional algorithms to validate these formulae, which are also useful for the demonstration of the minimality of cyclotomic numbers required.  
\end{abstract}
\section{Introduction}
Jacobi has many remarkable contributions to the field of mathematics like formulations of Jacobi symbol, Jacobi triple product, Jacobian, Jacobi elliptic functions etc. Among these, Jacobi sums have appeared as one of the most important findings. Many people have attempted to calculate the Jacobi sums of certain order in terms of the solutions of the corresponding Diophantine system. It has also been observed that the level of complexity in dealing with a Diophantine system increases with the increment in the concerned order, which has been observed as a challenge for the computation of Jacobi sums of higher orders.
\par
 The objective of this paper is to develop some fast computational algorithms for calculation of Jacobi sums.
\par
Let $p\in \mathbb{Z}$ be an odd prime and $q=p^r$ with $r\geq 1$ an integer. Let $e\geq 2$ be an integral divisor of $q-1$, then $q=ek+1$ for some positive integer $k$. Let $\gamma$ be a generator of the cyclic group $\mathbb{F}^{*}_{q}$. For $\zeta_e$ a primitive $e$-th root of unity, we define a multiplicative character $\chi_{e}$ of order $e$ on $\mathbb{F}^{*}_{q}$ by $\chi_{e}(\gamma)=\zeta_e$. We now extend $\chi_{e}$ to $\mathbb{F}_q$ by taking $\chi_{e}(0)=0$. For $0\leq i, j\leq e-1$, the Jacobi sums of order $e$ is defined by 
$$
J_e(i,j)= \sum_{v\in \mathbb{F}_q} \chi_{e}^i(v) \chi_{e}^j(v+1).
$$
For $0\leq a, b\leq e-1$, the cyclotomic numbers $(a, b)_e$ of order $e$ are defined as follows:
\begin{align*}
(a,b)_e:&= \#\{v\in\mathbb{F}_q|\chi_{e}(v)=\zeta_{e}^{a},\chi_{e}(v+1)=\zeta_{e}^{b}\}\\ &=\#\{v\in\mathbb{F}_q\setminus \{0,-1\}\mid {\rm ind}_{\gamma}v\equiv a \pmod e ,  {\ \rm ind}_{\gamma}(v+1)\equiv b \pmod e\}.
\end{align*}
In view of the definitions of Jacobi sums and cyclotomic numbers, over a given finite field $\mathbb{F}_q$, Jacobi sums (resp. cyclotomic numbers) of order $e$ mainly depend on two parameters, therefore these values could be naturally assembled into a matrix of order $e$.
Thus, the Jacobi sums $J_e(i,j)$ and the cyclotomic numbers $(a,b)_e$ are well connected  by the following relations \cite{Parnami2, Shirolkar1}:
\begin{equation} \label{01}
 \sum_a\sum_b(a,b)_e\zeta_e^{ai+bj}=J_e(i,j),
\end{equation} 
and
\begin{equation} \label{00}
 \sum_i\sum_j\zeta_e^{-(ai+bj)}J_e(i,j)=e^{2}(a,b)_e.
\end{equation}
Jacobi \cite{Jacobi1} introduced the Jacobi sums of orders $3, \ 4$ and $7$. In the literature, problem concerning to computation of Jacobi sums of a particular order also has been studied as the possible minimization in relations between Jacobi sums and cyclotomic numbers of that order, like for $e\leq 6$, $e=8,\ 10$ and $12$ the relations were established by Dickson \cite{Dickson1}. Later, he also studied the relations for $e=9,\ 14,\ 15,\ 16, \ 18,\ 20,\ 22,\ 24$, which are recorded in \cite{Dickson2,Dickson3}. Muskat \cite{Muskat1} established the relation of order $12$ in terms of the fourth root of unity to resolve the sign ambiguity which was occuring in the Dickson's relation and here he also studied for $e = 30$. He \cite{Muskat2} also provided complete results for order $14$. Complete methods of $e = 16$ and $20$ exist in Whiteman \cite{Whiteman5} and Muskat \cite{Muskat3} respectively. In fact Western \cite{Western1} determined Jacobi sums of orders $8,\ 9, \ 16$ earlier than Dickson. Baumert and Fredrickson \cite{Baumert1} gave corrections and removed the sign ambiguity in the Dicksons's work for $e=9, \ 18$. Zee \cite{Zee1,Zee2} found relations for orders $13,\ 22$ and $60$. Zee and Muskat \cite{Muskat4} provided the relations for $e=21,\ 28,\ 39,\ 55$ and $56$. Berndt and Evans \cite{Berndt2} obtained sums of orders $3,\ 4, \ 6,\ 8,\ 12,\ 20$ and $24$ and they also determined sums of orders $5,\ 10$ and $16$ in \cite{Berndt3}. Parnami, Agarwal and Rajwade \cite{Parnami2} obtained certain relationships for Jacobi sums of odd prime orders upto $19$. Furthermore, Katre and Rajwade \cite{Katre3} extended their work for Jacobi sums of general odd prime orders.
\par
Over the most recent couple of years, fast computation of Jacobi sums is one of the essential enthusiasm among researchers, in perspective of its applications to primality testing, cryptosystems and so forth \cite{Adleman1,Buhler1,Leung1,Mihailescu1,Mihailescu2}. As illustrated in \cite{Ireland1}, Jacobi sums could be used for estimating the number of integral solutions to certain congruences such as $x^{3}+y^{3}\equiv 1 \pmod p$. These estimates played a key role in the advancement of Weil conjectures \cite{Weil1}. Jacobi sums could be used for the determination of a number of solutions of diagonal equations over finite fields \cite{Berndt1}. 
\par
In perspective of equations (\ref{01}) and (\ref{00}), to determine a Jacobi sum  of order $e$, one needs to determine all the cyclotomic numbers of order $e$. The cyclotomic numbers of order $l^{2}$ have  been formulated by Shirolkar and Katre \cite{Shirolkar1}, where as Ahmed, Tanti and Hoque \cite{Helal1} have established the formula for $e=2l^{2}$. 
\par 
In this paper, we give algorithms for fast computation of Jacobi sums of orders $l^2$ and $2l^{2}$ with $l$ a prime $\geq 3$. The idea used behind this paper is to compute all the Jacobi sums of a particular order in terms of the minimal number of cyclotomic numbers of that order as hinted from \cite{Helal2}. Here we derive explicit expressions for Jacobi sums of orders $l^2$ and $2l^{2}$ in terms of the minimal numbers of cyclotomic numbers of orders $l^2$ and $2l^{2}$ respectively. We implement two algorithms (see algorithms \ref{algo1} and \ref{algo2}) to validate these expressions for the minimality. The implementation of algorithms has been carried out at a high performance computing lab in Department of Computer Sciences and Technology, Central University of Jharkhand.
\par
The paper is organized as follows: Section \ref{section2} presents some well known properties of cyclotomic numbers of order $e$. Section \ref{section3} presents algorithms for equality of cyclotomic numbers of orders $l^2$ and $2l^2$. Section \ref{section4} contains the expressions of Jacobi sums of orders $l^2$ \& $2l^2$. The fast computational algorithms for Jacobi sums are in section \ref{section5}. Finally, a brief conclusion is reflected in section \ref{Conclusion}.  

\section{Some useful expressions}\label{section2}
It is clear that $(a,b)_{e}=(a' ,b')_{e}$ whenever $a\equiv a'\pmod{e}$ and $b\equiv b'\pmod{e}$ where as $(a,b)_{e}= (e-a,b-a)_{e}$. These imply the following identities:
\begin{equation}\label{2.1}
(a,b)_{e}=\begin{cases}
 (b,a)_{e}\hspace*{1.751cm} \text{ if }  k \text{ is  even  or  } q=2^r,\\
(b+\frac{e}{2},a+\frac{e}{2})_{e}\hspace*{0.4cm} \text{ otherwise}.
\end{cases}. 
\end{equation}
Applying these facts, it is easy to see that 
\begin{equation} \label{2.2}
\sum_{a=0}^{e-1}\sum_{b=0}^{e-1}(a,b)_{e}=q-2,
\end{equation}
and
\begin{equation} \label{2.3}
\sum_{b=0}^{e-1}(a,b)_{e}=k-n_{a},
\end{equation}
where $n_{a}$ is given by 
\begin{equation*}
n_{a}=\begin{cases}
1 \quad \text{ if } a=0, 2\mid k \text{ or if } a=\frac{e}{2}, 2\nmid k;\\
0 \quad \text{ otherwise }.
\end{cases}
\end{equation*}
\section{Algorithms for equality of cyclotomic numbers}\label{section3}
Solution of cyclotomy of order $e$, does not require to determine all the cyclotomic numbers
of order $e$ \cite{Helal2}. The objective is to avoid the redundancy in calculations by dividing the cyclotomic numbers into classes as hinted from \cite{Helal2}, which certainly boost up the overall efficiency.
\par
This section, presents two algorithms which shows the equality relations of cyclotomic numbers of orders $l^{2}$ and $2l^{2}$ respectively. These algorithms exactly determine which cyclotomic numbers are enough for the determination of all the Jacobi sums of orders $l^{2}$ and $2l^{2}$ respectively. Thus, it helps us to faster the computation of these Jacobi sums. Also, these algorithms play a major role to validate the expressions for Jacobi sums of orders $l^{2}$ and $2l^{2}$ in terms of the minimum number of cyclotomic numbers of orders $l^{2}$ and $2l^{2}$ respectively. The expressions in Theorem \ref{thm1} gets validated by \textquoteleft Algorithm \ref{algo1}\textquoteright, and those in Theorems \ref{thm2} and \ref{thm3} get validated by \textquoteleft Algorithm \ref{algo2}\textquoteright.  
\begin{algorithm}[H]
\caption{Determination of equality of cyclotomic numbers of order $l^{2}$.}\label{algo1}
\begin{algorithmic}[1]
\vspace{.5cm}
\State START
\State Declare Integer variable $p,q,r,l,i,j,i1,j1,a1,b1,flag$.
\State INPUT $l$
\If{$l$ is not a prime or less than $3$}
\State goto $3$
\Else
\State $e=l^{2}$
\EndIf
\State Declare an array of size $e\times e$, where each element of array is $2$ tuple structure (i.e. ordered pair of $(a,b)$, where $a$ and $b$ are integers).
\State INPUT $q$ 
\For{$p=$ all prime number within $2$ to $q$}
\For{all $r$ within $1$ to $q$}
\If{$q$ is not equal to $p^{r}$}
\State goto $10$
\ElsIf{$(q-1)\%e==0$}
\State goto $20$
\Else
\State goto $10$
\EndIf
\EndFor
\EndFor
\algstore{c}
\end{algorithmic}
\end{algorithm}	
\newpage
\begin{algorithm}
\begin{algorithmic}[1]
\algrestore{c}
\vspace{.5cm}
\For {$i1=0$ to $e-1$}
\For {$j1=0$ to $e-1$}
\State int $a1=$ value of $a$ at current array index i.e. $a1=arr[i1][j1].a$
\State int $b1=$ value of $b$ at current array index i.e. $b1=arr[i1][j1].b$
\State set flag of current element of array to $0$ i.e. lock the element which is updated once $arr[i1][j1].flag=0$ 
\EndFor
\EndFor
\For {$i=0$ to $e-1$}
\For {$j=0$ to $e-1$}
\If{flag is $1$ i.e. element has not been updated yet}
\If {$a$ is equal to $b1\%e$ (when $b1\geq 0$) and $b1+e$ (when $b1<0$) AND $b$ is equal to $a1\%e$ (when $a1\geq 0$) and $a1+e$ (when $a1<0$)}
\State $a=a1$, $b=b1$ and $flag=0$
\EndIf
\If {$a$ is equal to $(a1-b1)\%e$ (when $(a1-b1)\geq 0$) and $(a1-b1)+e$ (when $(a1-b1)<0$) AND $b$ is equal to $(-b1)\%e$ (when $(-b1)\geq 0$) and $(-b1)+e$ (when $(-b1)<0$)}
\State $a=a1$, $b=b1$ and $flag=0$
\EndIf
\If {$a$ is equal to $(b1-a1)\%e$ (when $(b1-a1)\geq 0$) and $(b1-a1)+e$ (when $(b1-a1)<0$) AND $b$ is equal to $(-a1)\%e$ (when $(-a1)\geq 0$) and $(-a1)+e$ (when $(-a1)<0$)}
\State $a=a1$, $b=b1$ and $flag=0$
\EndIf
\If {$a$ is equal to $(-a1)\%e$ (when $(-a1)\geq 0$) and $(-a1)+e$ (when $(-a1)<0$) AND $b$ is equal to $(b1-a1)\%e$ (when $(b1-a1)\geq 0$) and $(b1-a1)+e$ (when $(b1-a1)<0$)}
\State $a=a1$, $b=b1$ and $flag=0$
\EndIf
\If {$a$ is equal to $(-b1)\%e$ (when $(-b1)\geq 0$) and $(-b1)+e$ (when $(-b1)<0$) AND $b$ is equal to $(a1-b1)\%e$ (when $(a1-b1)\geq 0$) and $(a1-b1)+e$ (when $(a1-b1)<0$)}
\State $a=a1$, $b=b1$ and $flag=0$
\EndIf
\EndIf
\EndFor
\EndFor	
\end{algorithmic}
\end{algorithm}	
\begin{algorithm}[H]
\caption{Determination of equality of cyclotomic numbers of order $2l^{2}$.}\label{algo2}
\begin{algorithmic}[1]
\vspace{.5cm}
\State START
\State Declare Integer variable $p,q,r,l,i,j,k,i1,j1,a1,b1,flag$.
\State INPUT $l$
\If{$l$ is not a prime or less than $3$}
\State goto $3$
\Else
\State $e=2l^{2}$
\EndIf
\State Declare an array of size $e\times e$, where each element of array is $2$ tuple structure (i.e. ordered pair of $(a,b)$, where $a$ and $b$ are integers).
\State INPUT $q$ 
\For{$p=$ all prime number within $2$ to $q$}
\For{all $r$ within $1$ to $q$}
\If{$q$ is not equal to $p^{r}$}
\State goto $10$
\ElsIf{$(q-1)\%e==0$}
\State $k=(q-1)/e$
\Else
\State goto $10$
\If{ $k$ even}
\State goto $25$
\Else
\State goto $53$
\EndIf
\EndIf
\EndFor
\EndFor
\For {$i1=0$ to $e-1$}
\For {$j1=0$ to $e-1$}
\State int $a1=$ value of $a$ at current array index i.e. $a1=arr[i1][j1].a$
\State int $b1=$ value of $b$ at current array index i.e. $b1=arr[i1][j1].b$
\State set flag of current element of array to $0$ i.e. lock the element which is updated once $arr[i1][j1].flag=0$ 
\EndFor
\EndFor
\For {$i=0$ to $e-1$}
\For {$j=0$ to $e-1$}
\If{flag is $1$ i.e. element has not been updated yet}
\If {$a$ is equal to $b1\%e$ (when $b1\geq 0$) and $b1+e$ (when $b1<0$) AND $b$ is equal to $a1\%e$ (when $a1\geq 0$) and $a1+e$ (when $a1<0$)}
\State $a=a1$, $b=b1$ and $flag=0$
\EndIf
\If {$a$ is equal to $(a1-b1)\%e$ (when $(a1-b1)\geq 0$) and $(a1-b1)+e$ (when $(a1-b1)<0$) AND $b$ is equal to $(-b1)\%e$ (when $(-b1)\geq 0$) and $(-b1)+e$ (when $(-b1)<0$)}
\State $a=a1$, $b=b1$ and $flag=0$
\EndIf
\algstore{a}
\end{algorithmic}
\end{algorithm}
\newpage
\begin{algorithm}[H]
\begin{algorithmic}[1]
\algrestore{a}
\vspace{.5cm}
\If {$a$ is equal to $(b1-a1)\%e$ (when $(b1-a1)\geq 0$) and $(b1-a1)+e$ (when $(b1-a1)<0$) AND $b$ is equal to $(-a1)\%e$ (when $(-a1)\geq 0$) and $(-a1)+e$ (when $(-a1)<0$)}
\State $a=a1$, $b=b1$ and $flag=0$
\EndIf
\If {$a$ is equal to $(-a1)\%e$ (when $(-a1)\geq 0$) and $(-a1)+e$ (when $(-a1)<0$) AND $b$ is equal to $(b1-a1)\%e$ (when $(b1-a1)\geq 0$) and $(b1-a1)+e$ (when $(b1-a1)<0$)}
\State $a=a1$, $b=b1$ and $flag=0$
\EndIf
\If {$a$ is equal to $(-b1)\%e$ (when $(-b1)\geq 0$) and $(-b1)+e$ (when $(-b1)<0$) AND $b$ is equal to $(a1-b1)\%e$ (when $(a1-b1)\geq 0$) and $(a1-b1)+e$ (when $(a1-b1)<0$)}
\State $a=a1$, $b=b1$ and $flag=0$
\EndIf
\EndIf
\EndFor
\EndFor	
\For {$i1=0$ to $e-1$}
\For {$j1=0$ to $e-1$}
\State int $a1=$ value of $a$ at current array index i.e. $a1=arr[i1][j1].a$
\State int $b1=$ value of $b$ at current array index i.e. $b1=arr[i1][j1].b$
\State set flag of current element of array to $0$ i.e. lock the element which is updated once $arr[i1][j1].flag=0$ 
\EndFor
\EndFor
\For {$i=0$ to $e-1$}
\For {$j=0$ to $e-1$}
\If{flag is $1$ i.e. element has not been updated yet}
\If {$a$ is equal to $(b1+l^{2})\%e$ (when $(b1+l^{2})\geq 0$) and $(b1+l^{2})+e$ (when $(b1+l^{2})<0$) AND $b$ is equal to $(a1+l^{2})\%e$ (when $(a1+l^{2})\geq 0$) and $(a1+l^{2})+e$ (when $(a1+l^{2})<0$)}
\State $a=a1$, $b=b1$ and $flag=0$
\EndIf
\If {$a$ is equal to $(l^{2}+a1-b1)\%e$ (when $(l^{2}+a1-b1)\geq 0$) and $(l^{2}+a1-b1)+e$ (when $(l^{2}+a1-b1)<0$) AND $b$ is equal to $(-b1)\%e$ (when $(-b1)\geq 0$) and $(-b1)+e$ (when $(-b1)<0$)}
\State $a=a1$, $b=b1$ and $flag=0$
\EndIf
\If {$a$ is equal to $(l^{2}+b1-a1)\%e$ (when $(l^{2}+b1-a1)\geq 0$) and $(l^{2}+b1-a1)+e$ (when $(l^{2}+b1-a1)<0$) AND $b$ is equal to $(l^{2}-a1)\%e$ (when $(l^{2}-a1)\geq 0$) and $(l^{2}-a1)+e$ (when $(l^{2}-a1)<0$)}
\State $a=a1$, $b=b1$ and $flag=0$
\EndIf
\If {$a$ is equal to $(-a1)\%e$ (when $(-a1)\geq 0$) and $(-a1)+e$ (when $(-a1)<0$) AND $b$ is equal to $(b1-a1)\%e$ (when $(b1-a1)\geq 0$) and $(b1-a1)+e$ (when $(b1-a1)<0$)}
\State $a=a1$, $b=b1$ and $flag=0$
\EndIf
\algstore{b}
\end{algorithmic}
\end{algorithm}
\newpage
\begin{algorithm}
\begin{algorithmic}
\algrestore{b}
\vspace{.5cm}
\If {$a$ is equal to $(l^{2}-b1)\%e$ (when $(l^{2}-b1)\geq 0$) and $(l^{2}-b1)+e$ (when $(l^{2}-b1)<0$) AND $b$ is equal to $(a1-b1)\%e$ (when $(a1-b1)\geq 0$) and $(a1-b1)+e$ (when $(a1-b1)<0$)}
\State $a=a1$, $b=b1$ and $flag=0$
\EndIf
\EndIf
\EndFor
\EndFor	
\end{algorithmic}
\end{algorithm}

The above algorithms demonstrate that for the determination of Jacobi sums of orders $2l^{2}$ and $l^{2}$ with prime $l\geq 5$, it is adequate to determine $2l^2+{(2l^2-1)(2l^2-2)}/6$ and $l^2+{(l^2-1)(l^2-2)}/6$ cyclotomic numbers of orders $2l^2$ and $l^{2}$ respectively. However for $l=3$, it is sufficient to ascertain $64$ and $19$ cyclotomic numbers of orders $2l^2$ and $l^{2}$ respectively. Thus, it reduces the complexity of order $l^2$ to $l^{4}-\{l^{2}+{(l^{2}-1)(l^{2}-2)}/6\}$ for $l > 3$ and $l^{4}-19$ for $l = 3$ and of order $2l^{2}$ to $4l^{4}- \{2l^{2}+{(2l^{2}-1)(2l^{2}-2)}/6\}$  for $l > 3$ and $4l^{4}-64$ for $l = 3$. So, it could be easily observed that, for a large value of $l$, complexity for the determination of Jacobi sums reduces drastically.

\begin{table}[h!]
\centering
{\tiny{%
\begin{tabular}{@{}|c|c|c|c|@{}}
\toprule
\multirow{2}{*}{\begin{tabular}[c]{@{}c@{}}Corresponding \\ value of $l$\end{tabular}} & \multicolumn{3}{c|}{Order $l^2$} \\ \cmidrule(l){2-4} 
 & \begin{tabular}[c]{@{}c@{}}Required number of \\ cyclotomic numbers \\ need to determine\end{tabular} & \begin{tabular}[c]{@{}c@{}}Actual number of \\ cyclotomic numbers \\ need to determine\end{tabular} & \begin{tabular}[c]{@{}c@{}}Number of reduced\\ computations for \\ Jacobi sums\end{tabular} \\ \midrule
$l=3$ & 81 & 19 & 62 \\ \midrule
$l=5$ & 625 & 117 & 508 \\ \midrule
$l=7$ & 2401 & 425 & 1976 \\ \midrule
$l=11$ & 14641 & 2501 & 12140 \\ \midrule
$l=13$ & 28561 & 4845 & 23716 \\ \bottomrule
\end{tabular}%
}}
\vspace{0.1cm}
\caption{Complexity comparison of order $l^2$}
\label{tab:my-table}
\end{table}
\begin{table}[h]
\centering
{\tiny{%
\begin{tabular}{@{}|c|c|c|c|@{}}
\toprule
\multirow{2}{*}{\begin{tabular}[c]{@{}c@{}}Corresponding \\ value of $l$\end{tabular}} & \multicolumn{3}{c|}{Order $2l^2$} \\ \cmidrule(l){2-4} 
 & \begin{tabular}[c]{@{}c@{}}Required number of \\ cyclotomic numbers \\ need to determine\end{tabular} & \begin{tabular}[c]{@{}c@{}}Actual number of \\ cyclotomic numbers \\ need to determine\end{tabular} & \begin{tabular}[c]{@{}c@{}}Number of reduced\\ computations for \\ Jacobi sums\end{tabular} \\ \midrule
$l=3$ & 324 & 64 & 260 \\ \midrule
$l=5$ & 2500 & 442 & 2058 \\ \midrule
$l=7$ & 9604 & 1650 & 7954 \\ \midrule
$l=11$ & 58564 & 9882 & 48682 \\ \midrule
$l=13$ & 114244 & 19210 & 95034 \\ \bottomrule
\end{tabular}%
}}
\vspace{0.1cm}
\caption{Complexity comparison of order $2l^2$}
\label{tab:my-table 1}
\end{table}
\newpage
As illustrated in tables \ref{tab:my-table} and \ref{tab:my-table 1}, for $l=3$, naively summing the definition, we need to evaluate 81 and 324 numbers of cyclotomic numbers of orders $l^2$ and $2l^2$ respectively to determine the Jacobi sums of respective orders. But algorithms \ref{algo1} and \ref{algo2} validate that it is sufficient to evaluate only 19 and 64 numbers of cyclotomic numbers of order $9$ and $18$ respectively. Thus complexity of implementing algorithms for Jacobi sums reduces by 62 and 260 respectively. Also one can observe that, as the value of $l$ increases, the corresponding complexity reduces drastically. Consequently, efficiency of our implemented algorithms \ref{algo3}, \ref{algo4}, \ref{algo5} increases as value of $l$ rises. 

\section{Expressions for Jacobi sums in terms of cyclotomic numbers}\label{section4}
Here, we present the expressions for Jacobi sums of orders $l^2$ and $2l^2$, $l$ an odd prime, in terms of the minimum number of cyclotomic numbers of orders $l^2$ and $2l^2$ respectively.
\begin{theorem}\label{thm1}
Let $l$ be an odd prime and $p$ a prime. For some positive integers $r$ and $k$, let $q=p^r=l^2k+1$.\\
Then for $l\geq 5$
\begin{align}\label{exp5}
J_{l^{2}}(1,n)&=\sum_{a=0}^{l^{2}-1}\sum_{b=0}^{l^{2}-1}(a,b)_{l^{2}}\zeta_{l^{2}}^{a+bn}\nonumber\\ &= (0,0)_{l^{2}}+\sum_{b=1}^{l^{2}-1}(0,b)_{l^{2}}(\zeta_{l^{2}}^{bn}+\zeta_{l^{2}}^{b}+\zeta_{l^{2}}^{-b(n+1)})+\bigg\{\sum_{b=2}^{l^{2}-2}(1,b)_{l^{2}}+\sum_{b=4}^{l^{2}-3}(2,b)_{l^{2}}\nonumber\\ &\quad+\sum_{b=6}^{l^{2}-4}(3,b)_{l^{2}} 
+\dots+\sum_{b=(2l^{2}-2)/3}^{(2l^{2}-2)/3}((l^{2}-1)/3,b)_{l^{2}}\bigg\}\bigg(\zeta_{l^{2}}^{an+b}+\zeta_{l^{2}}^{a+bn}\nonumber\\ & \quad+\zeta_{l^{2}}^{a-b(n+1)}+\zeta_{l^{2}}^{an-b(n+1)}+\zeta_{l^{2}}^{bn-a(n+1)}  +\zeta_{l^{2}}^{b-a(n+1)}\bigg). 
\end{align}
and for $l=3$
\begin{align}\label{exp6}
J_{9}(1,n)&=\sum_{a=0}^{8}\sum_{b=0}^{8}(a,b)_{9}\zeta_{9}^{a+bn}\nonumber\\ &= (0,0)_{9}+ (3,6)_{9}(\zeta_{9}^{3+6n}+\zeta_{9}^{6+3n})+\sum_{b=1}^{8}(0,b)_{9}(\zeta_{9}^{bn}+\zeta_{9}^{b}+\zeta_{9}^{-b(n+1)})\nonumber\\&\quad+\bigg\{\sum_{b=2}^{7}(1,b)_{9}+\sum_{b=4}^{6}(2,b)_{9}\bigg\} \bigg(\zeta_{9}^{an+b}+\zeta_{9}^{a+bn}+\zeta_{9}^{a-b(n+1)}+\zeta_{9}^{an-b(n+1)}\nonumber\\ &\quad+\zeta_{9}^{bn-a(n+1)} +\zeta_{9}^{b-a(n+1)}\bigg). 
\end{align}
\end{theorem}
\begin{proof}
Cyclotomic numbers $(a, b)_{l^{2}}$ of order $l^{2}$  over $\mathbb{F}_{q}$ is defined in \cite{Shirolkar1} as:
\begin{center}
	$(a,b)_{l^{2}}:=\#\{v\in\mathbb{F}_q\setminus \{0,-1\}\mid {\rm ind}_{\gamma}v\equiv a \pmod {l^{2}} ,  {\ \rm ind}_{\gamma}(v+1)\equiv b \pmod {l^{2}}\}.$
\end{center}
In \cite{Shirolkar1}, it was proved that $(a, b)_{e}$ has the following properties:  
\begin{equation} \label{equ11}
(a, b)_{e} \quad = \quad (e-a, b-a)_{e}, \qquad (a, b)_{e} \quad = \quad (b, a)_{e}, \quad 2|k\ \text{or}\ q=2^{r}
\end{equation} 
and 
\begin{equation}\label{equ12}
(a, b)_{e} \quad = \quad (b+e/2, a+e/2)_{e} , \qquad (a, b)_{e} \quad = \quad (e-a, b-a)_{e}, \quad otherwise.
\end{equation} 
For $q=l^{2}k+1$ for some positive integer $k$, it is natural to see that $k$ is always even. Now  to permute only (\ref{equ11}), one gets
\begin{equation} \label{equ13}
(a,b)_{l^{2}}=(b,a)_{l^{2}}=(a-b,-b)_{l^{2}}=(b-a,-a)_{l^{2}}=(-a,b-a)_{l^{2}}=(-b,a-b)_{l^{2}}.
\end{equation}
We know that $(a,b)_{l^{2}}$ and $J_{l^{2}}(1,n)$ are well connected  by the following relations:
\begin{equation} \label{equ14}
J_{l^{2}}(1,n)=\sum_{a=0}^{l^2-1}\sum_{b=0}^{l^2-1}(a,b)_{l^{2}}\zeta_{l^{2}}^{a+bn}
\end{equation}
Thus by (\ref{equ13}), cyclotomic numbers $(a,b)_{l^{2}}$ of order $l^{2}$ partition into group of classes. For prime $l\geq 5$, cyclotomic numbers $(a,b)_{l^{2}}$ of order $l^{2}$ forms classes of singleton, three and six elements. $(0,0)_{l^{2}}$ form singleton class, $(-a,0)_{l^{2}}$, $(a,a)_{l^{2}}$, $(0,-a)_{l^{2}}$ forms a classes of three elements for every $1\leq a\leq {l^{2}}-1 \pmod {{l^{2}}}$ and rest of the cyclotomic numbers forms classes of six elements. For $l=3$ there are classes of singleton, second, three and six elements. The exception is $(6,3)_{9}=(3,6)_{9}$ which is grouped into a class of two elements. Hence expression (\ref{exp5}) and (\ref{exp6}) directly follows by the relation (\ref{equ14}).
\end{proof}
The following result gives an analogous expression for $l=3$. 
\begin{theorem}\label{thm2}
	Let $p$ be a prime. For some positive integers $r$ and $k$, let $q=p^r=18k+1$. 
	
	Then for $2|k$ or $q=2^{r}$,
	\begin{align}\label{exp1}
	J_{18}(1,n)&=\sum_{a=0}^{17}\sum_{b=0}^{17}(a,b)_{18}\zeta_{18}^{a+bn}\nonumber \\ &= (0,0)_{18}+(6,12)_{18}(\zeta_{18}^{6+12n}+\zeta_{18}^{12+6n})+\sum_{b=1}^{17}(0,b)_{18}(\zeta_{18}^{bn}\nonumber\\ & \quad+\zeta_{18}^{b}+\zeta_{18}^{-b(n+1)}) +\bigg\{ \sum_{b=2}^{16}(1,b)_{18}+\sum_{b=4}^{15}(2,b)_{18} +\sum_{b=6}^{14}(3,b)_{18}\nonumber\\ & \quad+\sum_{b=8}^{13}(4,b)_{18} +\sum_{b=10}^{12}(5,b)_{18}\bigg\}\bigg\{\zeta_{18}^{an+b}+\zeta_{18}^{a+bn} +\zeta_{18}^{a-b(n+1)}\nonumber \\ & \quad+\zeta_{18}^{an-b(n+1)}+\zeta_{18}^{bn-a(n+1)}+\zeta_{18}^{b-a(n+1)}\bigg\}. 
	\end{align}
	and for $2\nmid k$ and $q\neq 2^{r}$, 
	\begin{align}\label{exp2}
	J_{18}(1,n)&=\sum_{a=0}^{17}\sum_{b=0}^{17}(a,b)_{18}\zeta_{18}^{a+bn} \nonumber \\ &= (0,9)_{18}\zeta_{18}^{9n}+(6,3)_{18}(\zeta_{18}^{6+3n}+\zeta_{18}^{12+15n})+\bigg\{\sum_{b=0}^{8}(0,b)_{18}+\sum_{b=10}^{17}(0,b)_{18}\bigg\}\nonumber\\ & \quad\bigg\{\zeta_{18}^{bn}+\zeta_{18}^{9(n+1)+b}+\zeta_{18}^{9-b(n+1)}\bigg\}+\bigg\{\sum_{b=0}^{8}(1,b)_{18} +\sum_{b=12}^{17}(1,b)_{18} +\sum_{b=0}^{7}(2,b)_{18}\nonumber\\ & \quad+\sum_{b=14}^{17}(2,b)_{18} +\sum_{b=0}^{6}(3,b)_{18} +\sum_{b=16}^{17}(3,b)_{18} +\sum_{b=0}^{5}(4,b)_{18} +\sum_{b=1}^{3}(5,b)_{18}\bigg\}\nonumber \\ & \quad\bigg\{\zeta_{18}^{a+bn}+\zeta_{18}^{an+b+9(n+1)}+\zeta_{18}^{9+a-b(n+1)}+\zeta_{18}^{(n+1)9+b-a(n+1)}+\zeta_{18}^{bn-a(n+1)}\nonumber \\ & \quad+\zeta_{18}^{9+an-b(n+1)}\bigg\}. 
	\end{align}
\end{theorem}
\begin{proof}
We recall the following from \cite{Helal1}:
\begin{center}
	$(a,b)_{18}:=\#\{v\in\mathbb{F}_q\setminus \{0,-1\}\mid {\rm ind}_{\gamma}v\equiv a \pmod {18} ,  {\ \rm ind}_{\gamma}(v+1)\equiv b \pmod {18}\}.$
\end{center}
The following properties were derived in \cite{Helal1}:
\begin{equation} \label{equ1}
(a, b)_{18} \quad = \quad (18-a, b-a)_{18}, \qquad (a, b)_{18} \quad = \quad (b, a)_{18}, \quad 2|k\ \text{or}\ q=2^{r},
\end{equation} 
and 
\begin{equation}\label{equ2}
(a, b)_{18} \quad = \quad (b+9, a+9)_{18} , \qquad (a, b)_{18} \quad = \quad (18-a, b-a)_{18}, \quad otherwise.
\end{equation} 
We permute (\ref{equ1}) and (\ref{equ2}) to get
\begin{equation} \label{equ3}
(a,b)_{18}=(b,a)_{18}=(a-b,-b)_{18}=(b-a,-a)_{18}=(-a,b-a)_{18}=(-b,a-b)_{18}
\end{equation}
and
\begin{align} \label{equ4}
\nonumber &(a,b)_{18}=(b+9,a+9)_{18}=(9+a-b,-b)_{18}=(9+b-a,9-a)_{18}\\ & =(-a,b-a)_{18} =(9-b,a-b)_{18}
\end{align}
respectively.
\\
It is easy to see that $(a,b)_{18}$ and $J_{18}(1,n)$ are well-connected  by the following:
\begin{equation} \label{equ5}
J_{18}(1,n)=\sum_{a=0}^{17}\sum_{b=0}^{17}(a,b)_{18}\zeta_{18}^{a+bn}.
\end{equation}
Thus by (\ref{equ3}) and (\ref{equ4}), cyclotomic numbers $(a,b)_{18}$ of order $18$ partition into classes. If $2|k$ or $q=2^{r}$, (\ref{equ3}) gives classes of singleton, two, three and six elements. $(0,0)_{18}$ forms a singleton class, $(-a,0)_{18}$, $(a,a)_{18}$, $(0,-a)_{18}$ forms classes of three elements for every $1\leq a\leq 17 \pmod {18}$, $(6,12)_{18}=(12,6)_{18}$ which is grouped into a class of two elements and rest of the cyclotomic numbers forms classes of six elements. Hence expression (\ref{exp1}) directly follows from the relation (\ref{equ5}).
\\
Now if neither $2|k$ nor $q=2^{r}$, (\ref{equ4}) gives classes of singleton, two, three and six elements. $(0,9)_{18}$ forms a singleton class, $(0,a)_{18}$, $(a+9,9)_{18}$, $(9-a,-a)_{18}$ forms classes of three elements for every $0\leq a\neq 9 \leq 17 \pmod {18}$, $(6,3)_{18}=(12,15)_{18}$ which is grouped into a class of two elements and rest of the cyclotomic numbers forms classes of six elements. Hence expression (\ref{exp2}) directly follows by the relation (\ref{equ5}).
\end{proof}
\vspace{0.3cm}
\begin{remark}
$\#\bigg\{(0,0)_{18}+(6,12)_{18}+\sum_{b=1}^{17}(0,b)_{18}+ \sum_{b=2}^{16}(1,b)_{18}+\sum_{b=4}^{15}(2,b)_{18} +\sum_{b=6}^{14}(3,b)_{18}+\sum_{b=8}^{13}(4,b)_{18} +\sum_{b=10}^{12}(5,b)_{18}\bigg\}=\#\bigg\{(0,9)_{18}+(6,3)_{18}+\sum_{b=0}^{8}(0,b)_{18}+\sum_{b=10}^{17}(0,b)_{18}+\sum_{b=0}^{8}(1,b)_{18} +\sum_{b=12}^{17}(1,b)_{18} +\sum_{b=0}^{7}(2,b)_{18}+\sum_{b=14}^{17}(2,b)_{18} +\sum_{b=0}^{6}(3,b)_{18} +\sum_{b=16}^{17}(3,b)_{18} +\sum_{b=0}^{5}(4,b)_{18} +\sum_{b=1}^{3}(5,b)_{18}\bigg\}$.
\end{remark}
\vspace{0.3cm}
\begin{remark}
If $2|k$ or $q=2^{r}$, then the sum of all $(a,b)_{18}$, $0\leq a,b \leq 17$ is equal to $\bigg\{(0,0)_{18}+2(6,12)_{18}+3\sum_{b=1}^{17}(0,b)_{18}+6\bigg( \sum_{b=2}^{16}(1,b)_{18}+\sum_{b=4}^{15}(2,b)_{18} +\sum_{b=6}^{14}(3,b)_{18}+\sum_{b=8}^{13}(4,b)_{18} +\sum_{b=10}^{12}(5,b)_{18}\bigg)\bigg\}=q-2$.
\end{remark}
\vspace{0.3cm}
\begin{remark}
$2\nmid k$ and $q\neq 2^{r}$, then the sum of all $(a,b)_{18}$, $0\leq a,b \leq 17$ is equal to $\bigg\{(0,9)_{18}+2(6,3)_{18}+3\bigg(\sum_{b=0}^{8}(0,b)_{18}+\sum_{b=10}^{17}(0,b)_{18}\bigg)+6\bigg(\sum_{b=0}^{8}(1,b)_{18} +\sum_{b=12}^{17}(1,b)_{18} +\sum_{b=0}^{7}(2,b)_{18}+\sum_{b=14}^{17}(2,b)_{18} +\sum_{b=0}^{6}(3,b)_{18} +\sum_{b=16}^{17}(3,b)_{18} +\sum_{b=0}^{5}(4,b)_{18} +\sum_{b=1}^{3}(5,b)_{18}\bigg)\bigg\}=q-2$.
\end{remark}
\newpage
\begin{theorem}\label{thm3}
Let $l\geq 5$ and $p$ be primes. For some positive integers $r$ and $k$, let $q=p^r=2l^2k+1$. 
\par
Then for $2|k$ or $q=2^{r}$,
\begin{align}\label{exp3}
J_{2l^{2}}(1,n)&=\sum_{a=0}^{2l^{2}-1}\sum_{b=0}^{2l^{2}-1}(a,b)_{2l^{2}}\zeta_{2l^{2}}^{a+bn}\nonumber \\ &= (0,0)_{2l^{2}}+\sum_{b=1}^{2l^{2}-1}(0,b)_{2l^{2}}(\zeta_{2l^{2}}^{bn}+\zeta_{2l^{2}}^{b}+\zeta_{2l^{2}}^{-b(n+1)})+\bigg\{\sum_{b=2}^{2l^{2}-2}(1,b)_{2l^{2}} \nonumber\\ & \quad+\sum_{b=4}^{2l^{2}-3}(2,b)_{2l^{2}}+\sum_{b=6}^{2l^{2}-4}(3,b)_{2l^{2}}
+\dots+\sum_{b=(4l^{2}-4)/3}^{(4l^{2}-4)/3+1}((2l^{2}-2)/3,b)_{2l^{2}}\bigg\} \nonumber\\ & \quad \bigg\{\zeta_{2l^{2}}^{an+b}+\zeta_{2l^{2}}^{a+bn}+\zeta_{2l^{2}}^{a-b(n+1)}+\zeta_{2l^{2}}^{an-b(n+1)}+\zeta_{2l^{2}}^{bn-a(n+1)} +\zeta_{2l^{2}}^{b-a(n+1)}\bigg\}. 
\end{align}

and for $2\nmid k$ and $q\neq 2^{r}$,

\begin{align}\label{exp4}
J_{2l^{2}}(1,n)&=\sum_{a=0}^{2l^{2}-1}\sum_{b=0}^{2l^{2}-1}(a,b)_{2l^{2}}\zeta_{2l^{2}}^{a+bn}\nonumber\\ &= \nonumber (0,l^{2})_{2l^{2}} \zeta_{2l^{2}}^{l^{2}n}+\bigg\{\sum_{b=0}^{l^{2}-1}(0,b)_{2l^{2}}+\sum_{b=l^{2}+1}^{2l^{2}-1}(0,b)_{2l^{2}}\bigg\}\bigg\{\zeta_{2l^{2}}^{bn}+\zeta_{2l^{2}}^{b+(n+1)l^{2}}\nonumber\\ & \quad+\zeta_{2l^{2}}^{l^{2}-b(n+1)}\bigg\}+\bigg\{\sum_{b=0}^{l^{2}-1}(1,b)_{2l^{2}}+\sum_{b=l^{2}+3}^{2l^{2}-1}(1,b)_{2l^{2}} +\sum_{b=0}^{l^{2}-2}(2,b)_{2l^{2}}\nonumber\\ & \quad+\sum_{b=l^{2}+5}^{2l^{2}-1}(2,b)_{2l^{2}} +\sum_{b=0}^{l^{2}-3}(3,b)_{2l^{2}} +\sum_{b=l^{2}+7}^{2l^{2}-1}(3,b)_{2l^{2}}+\dots \nonumber\\  &\quad+ \sum_{b=0}^{l^{2}-(l^{2}-3)/2}(((l^{2}-1)/2)-1,b)_{2l^{2}}+\sum_{b=2l^{2}-2}^{2l^{2}-1}(((l^{2}-1)/2)-1,b)_{2l^{2}}\nonumber \\
&\quad+\sum_{b=1}^{(l^{2}-3)/2}((l^{2}-1)/2+1,b)_{2l^{2}}+\sum_{b=3}^{((l^{2}-3)/2)-1}(((l^{2}-1)/2)+2,b)_{2l^{2}}\nonumber\\ &\quad+\sum_{b=5}^{((l^{2}-3)/2)-2}(((l^{2}-1)/2)+3,b)_{2l^{2}}+\sum_{b=7}^{((l^{2}-3)/2)-3}(((l^{2}-1)/2)+4,b)_{2l^{2}}\nonumber\\ &\quad +\dots+\sum_{b=((l^{2}-3)/2)-((l^{2}-7)/6)-1}^{((l^{2}-3)/2)-((l^{2}-7)/6)}(((l^{2}-1)/2)+((l^{2}-1)/6),b)_{2l^{2}}\nonumber
\end{align}
\begin{align}
&\quad+\sum_{b=0}^{(l^{2}+1)/2}((l^{2}-1)/2,b)_{2l^{2}}\bigg\}\bigg\{\zeta_{2l^{2}}^{a+bn}+\zeta_{2l^{2}}^{b+(n+1)l^{2}+an}+\zeta_{2l^{2}}^{l^{2}+a-b(n+1)}\nonumber\\ &\quad+\zeta_{2l^{2}}^{b+(n+1)l^{2}-a(n+1)}+\zeta_{2l^{2}}^{b-a(n+1)}+\zeta_{2l^{2}}^{l^{2}+an-b(n+1)}\bigg\}.
\end{align}
\end{theorem}
\begin{proof}
The cyclotomic numbers $(a, b)_{2l^{2}}$ of order $2l^{2}$  over $\mathbb{F}_{q}$ is defined in \cite{Helal1} as follows:
\begin{center}
	$(a,b)_{2l^{2}}:=\#\{v\in\mathbb{F}_q\setminus \{0,-1\}\mid {\rm ind}_{\gamma}v\equiv a \pmod {2l^{2}} ,  {\ \rm ind}_{\gamma}(v+1)\equiv b \pmod {2l^{2}}\}.$
\end{center}
We now recall the following properties of $(a,b)_{2l^{2}}$ from \cite{Helal1}:
\begin{equation} \label{equ6}
(a, b)_{2l^{2}} \quad = \quad (2l^{2}-a, b-a)_{2l^{2}}, \qquad (a, b)_{2l^{2}} \quad = \quad (b, a)_{2l^{2}}, \quad \text{if}\ k\ \text{is even\ or}\ q=2^{r}
\end{equation} 
and 
\begin{equation}\label{equ7}
(a, b)_{2l^{2}} \quad = \quad (b+l^{2}, a+l^{2})_{2l^{2}} , \qquad (a, b)_{2l^{2}} \quad = \quad (2l^{2}-a, b-a)_{2l^{2}}, \quad otherwise.
\end{equation} 
By permuting (\ref{equ6}) and (\ref{equ7}), we obtain
\begin{equation} \label{equ8}
(a,b)_{2l^{2}}=(b,a)_{2l^{2}}=(a-b,-b)_{2l^{2}}=(b-a,-a)_{2l^{2}}=(-a,b-a)_{2l^{2}}=(-b,a-b)_{2l^{2}}
\end{equation}
and
\begin{align} \label{equ9}
\nonumber &(a,b)_{2l^{2}}=(b+l^{2},a+l^{2})_{2l^{2}}=(l^{2}+a-b,-b)_{2l^{2}}=(l^{2}+b-a,l^{2}-a)_{2l^{2}}\\ &=(-a,b-a)_{2l^{2}}  =(l^{2}-b,a-b)_{2l^{2}}.
\end{align}
We know that $(a,b)_{2l^{2}}$ and $J_{2l^{2}}(1,n)$ are well connected  by the following:
\begin{equation} \label{equ10}
J_{2l^{2}}(1,n)=\sum_{a=0}^{2l^2-1}\sum_{b=0}^{2l^2-1}(a,b)_{2l^{2}}\zeta_{2l^{2}}^{a+bn}
\end{equation}
Thus by (\ref{equ8}) and (\ref{equ9}), cyclotomic numbers $(a,b)_{2l^{2}}$ of order $2l^{2}$ partition into group of classes. If $2|k$ or $q=2^{r}$, (\ref{equ8}) gives classes of singleton, three and six elements. $(0,0)_{2l^{2}}$ forms a singleton class, $(-a,0)_{2l^{2}}$, $(a,a)_{2l^{2}}$, $(0,-a)_{2l^{2}}$ forms classes of three elements for every $1\leq a\leq {2l^{2}}-1 \pmod {{2l^{2}}}$ and rest of the cyclotomic numbers forms classes of six elements. Hence expression (\ref{exp3}) directly follows by the relation (\ref{equ10}).
\\
Now if neither $2|k$ nor $q=2^{r}$, (\ref{equ9}) forms classes of singleton, three and six elements. $(0,l^{2})_{2l^{2}}$ forms a singleton class, $(0,a)_{2l^{2}}$, $(a+l^{2},l^{2})_{2l^{2}}$, $(l^{2}-a,-a)_{2l^{2}}$ forms classes of three elements for every $0\leq a\neq l^2 \leq {2l^{2}}-1 \pmod {{2l^{2}}}$ and rest of the cyclotomic numbers forms classes of six elements. Hence expression (\ref{exp4}) directly follows by the relation (\ref{equ10}).
\end{proof}
\begin{remark}
$\#\bigg\{(0,0)_{2l^{2}}+\sum_{b=1}^{2l^{2}-1}(0,b)_{2l^{2}}+\sum_{b=2}^{2l^{2}-2}(1,b)_{2l^{2}}+\sum_{b=4}^{2l^{2}-3}(2,b)_{2l^{2}}+\sum_{b=6}^{2l^{2}-4}(3,b)_{2l^{2}} 
+\dots+\sum_{b=(4l^{2}-4)/3}^{(4l^{2}-4)/3+1}((2l^{2}-2)/3,b)_{2l^{2}}\bigg\}=\#\bigg\{(0,l^{2})_{2l^{2}} +\sum_{b=0}^{l^{2}-1}(0,b)_{2l^{2}}+\sum_{b=l^{2}+1}^{2l^{2}-1}(0,b)_{2l^{2}}+\sum_{b=0}^{l^{2}-1}(1,b)_{2l^{2}}+\sum_{b=l^{2}+3}^{2l^{2}-1}(1,b)_{2l^{2}} +\sum_{b=0}^{l^{2}-2}(2,b)_{2l^{2}}+\sum_{b=l^{2}+5}^{2l^{2}-1}(2,b)_{2l^{2}} +\sum_{b=0}^{l^{2}-3}(3,b)_{2l^{2}} +\sum_{b=l^{2}+7}^{2l^{2}-1}(3,b)_{2l^{2}}+\dots  +\sum_{b=0}^{l^{2}-(l^{2}-3)/2}(((l^{2}-1)/2)-1,b)_{2l^{2}}+\sum_{b=2l^{2}-2}^{2l^{2}-1}(((l^{2}-1)/2)-1,b)_{2l^{2}}+\sum_{b=1}^{(l^{2}-3)/2}((l^{2}+1)/2,b)_{2l^{2}}+\sum_{b=3}^{((l^{2}-3)/2)-1}(((l^{2}+1)/2)+1,b)_{2l^{2}}+\sum_{b=5}^{((l^{2}-3)/2)-2}(((l^{2}+1)/2)+2,b)_{2l^{2}}+\sum_{b=7}^{((l^{2}-3)/2)-3}(((l^{2}+1)/2)+3,b)_{2l^{2}}+\dots+\sum_{b=((l^{2}-3)/2)-((l^{2}-7)/6)-1}^{((l^{2}-3)/2)-((l^{2}-7)/6)}(((l^{2}+1)/2)+((l^{2}-7)/6),b)_{2l^{2}}+\sum_{b=0}^{(l^{2}+1)/2}((l^{2}-1)/2,b)_{2l^{2}}\bigg\}$.
\end{remark}
\vspace{0.3cm}
\begin{remark}
If $2|k$ or $q=2^{r}$, then the sum of all $(a,b)_{2l^{2}}$, $0\leq a,b \leq (2l^{2}-1)$ and $l\geq 5$ is equal to $\bigg\{(0,0)_{2l^{2}}+3\sum_{b=1}^{2l^{2}-1}(0,b)_{2l^{2}}+6\bigg(\sum_{b=2}^{2l^{2}-2}(1,b)_{2l^{2}}+\sum_{b=4}^{2l^{2}-3}(2,b)_{2l^{2}}+\sum_{b=6}^{2l^{2}-4}(3,b)_{2l^{2}} 
+\dots+\sum_{b=(4l^{2}-4)/3}^{(4l^{2}-4)/3+1}((2l^{2}-2)/3,b)_{2l^{2}}\bigg)\bigg\}=q-2$.
\end{remark}
\vspace{0.3cm}
\begin{remark}
If $2\nmid k$ and $q\neq 2^{r}$, then the sum of all $(a,b)_{2l^{2}}$, $0\leq a,b \leq (2l^{2}-1)$ and $l\geq 5$ is equal to $\bigg\{(0,l^{2})_{2l^{2}} +3\bigg(\sum_{b=0}^{l^{2}-1}(0,b)_{2l^{2}}+\sum_{b=l^{2}+1}^{2l^{2}-1}(0,b)_{2l^{2}}\bigg)+6\bigg(\sum_{b=0}^{l^{2}-1}(1,b)_{2l^{2}}+\sum_{b=l^{2}+3}^{2l^{2}-1}(1,b)_{2l^{2}} +\sum_{b=0}^{l^{2}-2}(2,b)_{2l^{2}}+\sum_{b=l^{2}+5}^{2l^{2}-1}(2,b)_{2l^{2}} +\sum_{b=0}^{l^{2}-3}(3,b)_{2l^{2}} +\sum_{b=l^{2}+7}^{2l^{2}-1}(3,b)_{2l^{2}}+\dots  +\sum_{b=0}^{l^{2}-(l^{2}-3)/2}(((l^{2}-1)/2)-1,b)_{2l^{2}}+\sum_{b=2l^{2}-2}^{2l^{2}-1}(((l^{2}-1)/2)-1,b)_{2l^{2}}+\sum_{b=1}^{(l^{2}-3)/2}((l^{2}+1)/2,b)_{2l^{2}}+\sum_{b=3}^{((l^{2}-3)/2)-1}(((l^{2}+1)/2)+1,b)_{2l^{2}}+\sum_{b=5}^{((l^{2}-3)/2)-2}(((l^{2}+1)/2)+2,b)_{2l^{2}}+\sum_{b=7}^{((l^{2}-3)/2)-3}(((l^{2}+1)/2)+3,b)_{2l^{2}}+\dots+\sum_{b=((l^{2}-3)/2)-((l^{2}-7)/6)-1}^{((l^{2}-3)/2)-((l^{2}-7)/6)}(((l^{2}+1)/2)+((l^{2}-7)/6),b)_{2l^{2}}+\sum_{b=0}^{(l^{2}+1)/2}((l^{2}-1)/2,b)_{2l^{2}}\bigg)\bigg\}=q-2$.
\end{remark}
\section{Fast computational algorithms for Jacobi sums}\label{section5} 
In a given finite field $\mathbb{F}_q$, Jacobi sums of order $e$, mainly depend on two parameters. Therefore, these values could be naturally assembled into a matrix of order $e$. For $e=l^{2}$ or $2l^{2}$, we know that by knowing the Jacobi sums $J_{e}(1,n), 0\leq n \leq (e-1)$, one can readily determine all the Jacobi sums of the respective order \cite{Helal1}. We implement algorithms for fast computation of $J_{l^{2}}(1,n)$ and $J_{2l^{2}}(1,n)$. 
\par
Throughout the algorithms, structure of individual term of a polynomial is by means of class structure \textquotedblleft term"(which is of the form $c\ \zeta_{e}^{d}; e= l^{2} \ \text{or} \ 2l^{2})$ and a different structure for a polynomial by means of a class structure \textquotedblleft poly".
Further \textquotedblleft poly $*p_{a}$" is a variable pointing to the resulting polynomial or say the master polynomial, \textquotedblleft poly $*p_{t}$" is again a variable pointing to keep a polynomial temporarily. The function add\_poly adds two polynomial expression or add a term with a polynomial.
\par 
Every time we declare a term, we need to assign the value of its coefficient and exponent. The function check\_sign\_of\_expo() will check the sign of each of the exponent of input expression and if it has been found to be negative then add $2l^{2}$ (for input expression of order $2l^{2}$) or $l^{2}$ (for input expression of order $l^{2}$) to the corresponding exponent. 
\par
Further function check\_break\_replace(), first checks whether the term has exponent greater than or equals to $l(l-1)$, if so then breaks the exponent into a power of $l(l-1)$ and then replaces each of the polynomial whose exponent is equal to $l(l-1)$ by polynomial $*p_{t}$, where \\
$*p_{t}= \begin{cases}
1-\zeta_{l^{2}}^{l}+\zeta_{l^{2}}^{2l}-\zeta_{l^{2}}^{3l}+\dots-\zeta_{l^{2}}^{l(l-2)} \ \ \ \ \ \ \ if  \ \text{expression is of order $l^{2}$}  ,\\
-1+\zeta_{2l^{2}}^{l}-\zeta_{2l^{2}}^{2l}+\zeta_{2l^{2}}^{3l}+\dots+\zeta_{2l^{2}}^{l(l-2)}\ \ if  \ \text{expression is of order $2l^{2}$}.
\end{cases}$
\begin{algorithm}[H]
\caption{Determination of Jacobi sums of order $l^2$}\label{algo3}
\begin{algorithmic}[1]
\State START
\State Input expression \ref{exp5}, if $l>3$; otherwise expression \ref{exp6} 
\State class term
\State \hspace{4.5mm} int coff
\State \hspace{4.5mm} int exp 
\State \hspace{4.5mm}  check\_sign\_of\_expo()
\State \hspace{4.5mm}  check\_break\_replace()
\State class poly
\State \hspace{4.5mm} term *t
\State \hspace{4.5mm} int degree
\State int main()
\State \hspace{4.5mm} Declare integer variable itr=0, c, n, min=1, max=$l^{2}-1$, a=0, b, $l$
\State \hspace{4.5mm} poly *$p_{a}$ //Declare master polynomial
\State \hspace{4.5mm} poly *$p_{t}$ // Declare a temporary polynomial
\State \hspace{4.5mm} INPUT n and $l$ 
\State \hspace{0.5mm}/*/////////////////////////////////////////////////*/
\State \hspace{4.5mm} INPUT value of (0,0) $\longrightarrow$ c
\State \hspace{9mm} term t1 //Declare a term
\State \hspace{9mm} t1.coff=c
\State \hspace{9mm} add\_poly($p_{a}$, t1) // add a term in polynomial
\State /* line number 22--32 required to evaluate particular case of order $l^2$, considering $l=3$ */
\State \hspace{4.5mm} INPUT value of (3,6) $\longrightarrow$ c
\State \hspace{9mm} term t2 
\State \hspace{9mm} t2.coff=c
\State \hspace{9mm} t2.exp=3+6n
\State \hspace{9mm} t2.check\_break\_replace ($p_{t}$)
\State \hspace{9mm} add\_poly($p_{a}$, $p_{t}$) // add two polynomial
\algstore{exp9}
\end{algorithmic}
\end{algorithm}
\begin{algorithm}[H]
\begin{algorithmic}
\algrestore{exp9}
\vspace{.5cm} 
\State \hspace{9mm} term t3
\State \hspace{9mm} t3.coff=c
\State \hspace{9mm} t3.exp=6+3n
\State \hspace{9mm} t3. check\_break\_replace ($p_{t}$)
\State \hspace{9mm} add\_poly($p_{a}$, $p_{t}$)
\State \hspace{0.5mm} /*/////////////////////////////////////////////////*/
\While {itr!=$(l^{2}-1)/3+1$}
\For {b=min to max and b++}
\State INPUT value of (a,b) $\longrightarrow$ c
\State Set limit=3
\State term t[6] // Declaring a term array 
\State t[0].coff=t[1].coff=t[2].coff=t[3].coff=t[4].coff=t[5].coff=c
\State t[0].exp=an+b
\State t[1].exp=a+bn
\State t[2].exp=a-b(n+1)
\If {min$>$1}
\State t[3].exp=an-b(n+1)
\State t[4].exp=bn-a(n+1)
\State t[5].exp=b-a(n+1)
\State Set limit=6
\EndIf
\For {i=0 to (limit-1) and i++}
\State t[i].check\_sign\_of\_exp()
\State t[i]. check\_break\_replace ($p_{t}$[i])
\State add\_poly($p_{a}$, $p_{t}$[i]) 
\EndFor
\EndFor 
\If {min is equal to 1}
\State min $\longleftarrow$ min + 1
\Else 
\State min $\longleftarrow$ min + 2 
\EndIf
\State max $\longleftarrow$ max - 1
\State itr $\longleftarrow$ itr + 1
\State a $\longleftarrow$ a + 1
\EndWhile
\end{algorithmic}
\end{algorithm}	
As discussed in section 3, classes of cyclotomic numbers of order $l^2$ differ for different values of $l$. For the chosen value of $l\geq5$, classes of cyclotomic numbers remain same but for $l=3$, it forms an additional class, which is a class of two elements.
Algorithm \ref{algo3} determine all the Jacobi sums of order $l^2$. If $l=3$, then initially line number 22-32 is required to evaluate and while loop would execute with a different conditional statement (which is itr!=3). 
\vspace{0.5cm}
\\
The condition in While loop should be itr!=3 because it forms two different classes of six elements and one class of three elements. 

\begin{algorithm}[H]
\caption{Determination of Jacobi sums of order $2l^2$, if either $2|k$ or $q=2^{r}$}\label{algo4}
\begin{algorithmic}[1]
\State START
\State Input expression \ref{exp3}, if $l>3$; otherwise expression \ref{exp1} 
\State class term
\State \hspace{4.5mm} int coff
\State \hspace{4.5mm} int exp 
\State \hspace{4.5mm} check\_sign\_of\_expo()
\State \hspace{4.5mm} check\_break\_replace()
\State class poly
\State \hspace{4.5mm} term *t
\State \hspace{4.5mm} int degree
\State int main()
\State \hspace{4.5mm} Declare integer variable itr=0, c, n, min=1, max=$2l^{2}-1$, a=0, b, $l$
\State \hspace{4.5mm} poly *$p_{a}$ //Declare master polynomial
\State \hspace{4.5mm} poly *$p_{t}$ // Declare a temporary polynomial
\State \hspace{4.5mm} INPUT n and $l$
\State \hspace{0.5mm}/*/////////////////////////////////////////////////*/
\State \hspace{4.5mm} INPUT value of (0,0) $\longrightarrow$ c
\State \hspace{9mm} term t1 //Declare a term
\State \hspace{9mm} t1.coff=c
\State \hspace{9mm} add\_poly($p_{a}$, t1) // add a term in polynomial
\State /* line number 22--32 required to evaluate particular case of order $2l^2$; i.e. $l=3$ */
\State \hspace{4.5mm} INPUT value of (6,12) $\longrightarrow$ c
\State \hspace{9mm} term t2 
\State \hspace{9mm} t2.coff=c
\State \hspace{9mm} t2.exp=6+12n
\State \hspace{9mm} t2.check\_break\_replace ($p_{t}$) 
\State \hspace{9mm} add\_poly($p_{a}$, $p_{t}$) // add two polynomial
\State \hspace{9mm} term t3
\State \hspace{9mm} t3.coff=c
\State \hspace{9mm} t3.exp=12+6n
\State \hspace{9mm} t3. check\_break\_replace ($p_{t}$)
\State \hspace{9mm} add\_poly($p_{a}$, $p_{t}$)
\State \hspace{0.5mm}/*/////////////////////////////////////////////////*/
\algstore{third}
\end{algorithmic}
\end{algorithm}
\begin{algorithm}[H]
\begin{algorithmic}
\algrestore{third}
\vspace{.5cm}
\While {itr!=$\frac{2l^{2}-2}{3}+1$}
\For {b=min to max and b++}
\State INPUT value of (a,b) $\longrightarrow$ c
\State Set limit=3
\State term t[6] // Declaring a term array 
\State t[0].coff=t[1].coff=t[2].coff=t[3].coff=t[4].coff=t[5].coff=c
\State t[0].exp=an+b
\State t[1].exp=a+bn
\State t[2].exp=a-b(n+1)
\If {min$>$1}
\State t[3].exp=an-b(n+1)
\State t[4].exp=bn-a(n+1)
\State t[5].exp=b-a(n+1)
\State Set limit=6
\EndIf
\For {i=0 to (limit-1) and i++}
\State t[i].check\_sign\_of\_exp()
\State t[i]. check\_break\_replace ($p_{t}$[i])
\State add\_poly($p_{a}$, $p_{t}$[i]) 
\EndFor
\EndFor 
\If {min is equal to 1}
\State min $\longleftarrow$ min + 1
\Else 
\State min $\longleftarrow$ min + 2 
\EndIf
\State max $\longleftarrow$ max - 1
\State itr $\longleftarrow$ itr + 1
\State a $\longleftarrow$ a + 1
\EndWhile
\end{algorithmic}
\end{algorithm}

Similarly, classes of cyclotomic numbers of order $2l^2$ differ for different values of $l$. For $l\geq5$, classes of cyclotomic numbers remain same but for $l=3$, it forms an additional class, which is a class of two elements. Algorithm \ref{algo4} implemented to determine all the Jacobi sums of order $2l^2$ under the assumption either $2|k$ or $q=2^{r}$. If $l=3$, then initially line number 22-32 is required to evaluate and while loop would execute with a different conditional statement (which is itr!=6). The condition in While loop should be itr!=6 because it forms five different classes of six elements and one class of three elements. 
\begin{algorithm}[H]
\caption{Determination of Jacobi sums of order $2l^2$, if either $2\nmid k$ or $q\neq 2^{r}$}\label{algo5}
\begin{algorithmic}[1]
\vspace{.5cm}
\State START
\State Input expression \ref{exp4}, if $l>3$; otherwise expression \ref{exp2} 
\State class term
\State \hspace{4.5mm} int coff
\State \hspace{4.5mm} int exp 
\State \hspace{4.5mm} check\_sign\_of\_expo()
\State \hspace{4.5mm} check\_break\_replace()
\State class poly
\State \hspace{4.5mm} term *t
\State \hspace{4.5mm} int degree
\State int main()
\State \hspace{3mm} Declare integer variable count1=0, count2=0, c, n, min1=0, max1=$l^{2}-1$, min2=$l^{2}+1$, max2=$2l^{2}-1$, a=0, b, $l$
\State \hspace{4.5mm} poly *$p_{a}$ //Declare master polynomial
\State \hspace{4.5mm} poly *$p_{t}$ // Declare a temporary polynomial
\State \hspace{4.5mm} INPUT n and $l$
\State \hspace{0.5mm}/*/////////////////////////////////////////////////*/
\State \hspace{4.5mm} INPUT value of (0,$l^{2}$) $\longrightarrow$ c
\State \hspace{9mm} term t1 //Declare a term
\State \hspace{9mm} t1.coff=c
\State \hspace{9mm} t1.exp=$l^{2}$n
\State \hspace{9mm} t1.check\_break\_replace ($p_{t}$) 
\State \hspace{9mm} add\_poly($p_{a}$, $p_{t}$) // add two polynomial
\State /* line number 24--34 required to evaluate particular case of order $2l^2$; i.e. $l=3$ */
\State INPUT value of (6,3) $\longrightarrow$ c
\State \hspace{4.5mm} term t2 
\State \hspace{4.5mm} t2.coff=c
\State \hspace{4.5mm} t2.exp=6+3n
\State \hspace{4.5mm} t2.check\_break\_replace ($p_{t}$) 
\State \hspace{4.5mm} add\_poly($p_{a}$, $p_{t}$) // add two polynomial
\State \hspace{4.5mm} term t3
\State \hspace{4.5mm} t3.coff=c
\State \hspace{4.5mm} t3.exp=12+15n
\State \hspace{4.5mm} t3. check\_break\_replace ($p_{t}$)
\State \hspace{4.5mm} add\_poly($p_{a}$, $p_{t}$)
\State \hspace{0.5mm}/*/////////////////////////////////////////////////*/
\While {count1!=$(l^{2}-1)/2+(l^{2}-1)/6+1$}
\If {count1!=$(l^{2}-1)/2+(l^{2}-1)/6+1$}
\For {b=min1 to max1 and b++}
\State INPUT value of (a,b) $\longrightarrow$ c
\State Set limit=3
\State term t[6] // Declaring a term array 
\State t[0].coff=t[1].coff=t[2].coff=t[3].coff=t[4].coff=t[5].coff=c
\State t[0].exp=a+bn
\State t[1].exp=an+b+9(n+1)
\State t[2].exp=9+a-b(n+1)
\algstore{fourth}
\end{algorithmic}
\end{algorithm}
\begin{algorithm}[H]
\begin{algorithmic}
\algrestore{fourth}
\vspace{.5cm}
\If {a$>$1}
\State t[3].exp=9(n+1)+b-a(n+1)
\State t[4].exp=bn-a(n+1)
\State t[5].exp=9+an-b(n+1)
\State Set limit=6
\EndIf
\For {i=0 to (limit-1) and i++}
\State t[i].check\_sign\_of\_exp()
\State t[i].check\_break\_replace ($p_{t}$[i])
\State add\_poly($p_{a}$, $p_{t}$[i]) 
\EndFor
\EndFor
\State count1++
\If{a==0}
\State max1=max1
\State min1=0
\ElsIf{count1$==(l^{2}-1)/2$}
\State max1=$(l^{2}-3)/2$
\State min1=1
\ElsIf{count1$>(l^{2}-1)/2$}
\State max1=max1-1
\State min1=min1+2
\Else
\State max1=max1-1
\State min1=0
\EndIf
\EndIf
\If {count2!=$(l^{2}-1)/2$}
\For {b=min2 to max2 and b++}
\State INPUT value of (a,b) $\longrightarrow$ c
\State Set limit=3
\State term t[6] // Declaring a term array 
\State t[0].coff=t[1].coff=t[2].coff=t[3].coff=t[4].coff=t[5].coff=c
\State t[0].exp=a+bn
\State t[1].exp=an+b+9(n+1)
\State t[2].exp=9+a-b(n+1)
\If {a$>$0}
\State t[3].exp=9(n+1)+b-a(n+1)
\State t[4].exp=bn-a(n+1)
\State t[5].exp=9+an-b(n+1)
\State Set limit=6
\EndIf
\For {i=0 to (limit-1) and i++}
\State t[i].check\_sign\_of\_exp()
\State t[i].check\_break\_replace ($p_{t}$[i])
\State add\_poly($p_{a}$, $p_{t}$[i]) 
\EndFor
\EndFor
\algstore{fifth}
\end{algorithmic}
\end{algorithm}
\newpage
\begin{algorithm}[H]
\begin{algorithmic}
\algrestore{fifth}
\vspace{.5cm}
\State min2=min2+2
\State count2++
\EndIf
\State a $\longleftarrow$ a + 1
\EndWhile
\end{algorithmic}
\end{algorithm}	

Algorithm \ref{algo4} implemented to determine all the Jacobi sums of order $2l^2$ under the assumption that neither $2|k$ nor $q=2^{r}$. If $l=3$, then initially line number 24-34 is required to evaluate and while loop would execute with a different conditional statement (which is count1!=6). The condition in While loop should be count1!=6 because it forms five different classes of six elements and one class of three elements. 
\section{Conclusion}\label{Conclusion}
In this article, we exhibited fast computational algorithms for determination of all the Jacobi sums of orders $l^2$ and $2l^2$ with $l\geq 3$ a prime. These algorithms were implemented in a High Performance Computing Lab. To increase the efficiency, we presented explicit expressions for Jacobi sums of orders $l^2$ and $2l^{2}$ in terms of the minimum number of cyclotomic numbers of respective orders, which has been utilized in implementing the algorithms. Also, we implemented two additional algorithms to validate the minimality of these expressions. 
\section*{Acknowledgment}
\noindent
The authors acknowledge Central University of Jharkhand, Ranchi, Jharkhand for providing necessary and excellent facilities to carry out this research.

\end{document}